\documentclass[letterpaper, 10 pt, conference]{ieeeconf}
\IEEEoverridecommandlockouts      
\pdfoutput=1
\usepackage{graphics} 
\usepackage{epsfig} 
\usepackage[cmex10]{amsmath}   
\usepackage{amsmath} 
\usepackage{amsfonts,amssymb,amsbsy}
\usepackage{latexsym,theorem}
\usepackage{epstopdf}
\usepackage{cite}
\usepackage{enumerate}

\newcommand{\Rset}{\mathbb{R}}
\newcommand{\Nset}{\mathbb{N}}
\newcommand{\cB}{\mathcal{B}}

\newcommand{\cD}{\mathcal{D}}

{\theorembodyfont{\slshape}\newtheorem{theorem}{Theorem}[section]}
{\theorembodyfont{\slshape}}
{\theorembodyfont{\slshape}\newtheorem{lemma}[theorem]{Lemma}}
{\theorembodyfont{\slshape}}
{\theorembodyfont{\slshape}}
{\theorembodyfont{\upshape}}
{\theorembodyfont{\upshape}\newtheorem{problem}[theorem]{Problem}}
{\theorembodyfont{\upshape}\newtheorem{remark}[theorem]{Remark}}
{\theorembodyfont{\upshape}}

\title{\LARGE \bf
An Algorithm for Solving Quadratic Optimization Problems with Nonlinear Equality Constraints 
}

\author{Tuan T. Nguyen, Mircea Lazar and Hans Butler
\thanks{The authors are with the Department of Electrical Engineering,
	Eindhoven University of Technology, P.O. Box 513, 5600
	MB Eindhoven, The Netherlands.
	E-mails: {\tt\small \{t.t.nguyen,m.lazar,h.butler\}@tue.nl}}%
}

\begin{document}

\maketitle
\thispagestyle{empty}
\pagestyle{empty}

\begin{abstract}
The classical method to solve a quadratic optimization problem with nonlinear equality constraints is to solve the Karush-Kuhn-Tucker (KKT) optimality conditions using Newton's method. This approach however is usually computationally demanding, especially for large-scale problems. This paper presents a new computationally efficient algorithm for solving quadratic optimization problems with nonlinear equality constraints. It is proven that the proposed algorithm converges locally to a solution of the KKT optimality conditions. Two relevant application problems, fitting of ellipses and state reference generation for electrical machines, are presented to demonstrate the effectiveness of the proposed algorithm.
\end{abstract}

\section{INTRODUCTION}
Quadratic optimization problems with nonlinear equality constraints arise very frequently in science and engineering. They have a wide range of applications in computer vision~\cite{Gander1994,Halif1998,Fitzgibbon1999,Yu2004}, mechatronics~\cite{MeekerThesis,RubenAutomatica2012,Carpiuc2015}, system identification~\cite{Pelt2001ACC,Yeredor2004}, etc. They are also the heart of advanced optimization solvers such as trust region method~\cite{More1993,Rendl1997} and sequential quadratic programming~\cite{Zhu2008SQP,Morales2011}.

The classical method to solve a quadratic optimization problem with nonlinear equality constraints is to solve the Karush-Kuhn-Tucker (KKT) optimality conditions~\cite{KuhnTucker1951} using the iterative Newton's method~\cite{Nocedalbook}. The disadvantage of this approach is that it introduces additional optimization variables, i.e. Lagrange multipliers, which increases the complexity of the problem. In addition, this approach requires evaluation and storing of the Hessian matrix.


To reduce the computational load, this paper aims for an optimization algorithm which bypasses the KKT optimality conditions and implements a direct iterative search for the optimum. The quadratic optimization problem with nonlinear equality constraints is first transformed into a least 2-norm problem of an underdetermined nonlinear system of equations. A well known method to find a solution of an underdetermined system of equations is the generalized Newton's method~\cite{Levin2001}. However, this method only searches for a feasible solution of the system of equations.


In this paper we propose an alternative algorithm which searches for the minimum 2-norm solution of an underdetermined system of equations. The update of the algorithm at every iteration is defined as an interpolation between the generalized Newton's algorithm update and a proposed update that aims at minimizing the 2-norm of the solution. It is proven that there exists a suitable interpolation coefficient such that the developed algorithm converges locally to a local minimum of the 2-norm solution of the underdetermined system of equations. 

Compared to the classical approach which solves the KKT conditions using Newton's method, the proposed algorithm does not introduce additional optimization variables and does not require evaluation of the Hessian matrix. Therefore, the proposed algorithm requires less computational cost and less amount of memory. This feature will be beneficial for large-scale problems and real-time applications where the time for computation is limited. Two benchmark application problems, fitting of ellipses and state reference generation for electrical machines, are presented to demonstrate the effectiveness of the proposed algorithm.



The remainder of this paper is organized as follows. Section~\ref{sec:notation} introduces the notation used in the paper. The optimization problem is formulated in Section~\ref{sec:probform}. Section~\ref{sec:algorithm} presents the proposed numerical optimization algorithm and proves the local convergence property of the algorithm. Examples are shown in Section~\ref{sec:example} to demonstrate the effectiveness of the algorithm. Section~\ref{sec:conclusions} summarizes the conclusions.

\section{NOTATION}\label{sec:notation}
Let $\Nset$ denote the set of natural numbers, $\mathbb{R}$ denote the set of real numbers, $\mathbb{R}_{>0}$ denote the set of positive real numbers. The notation $\mathbb{R}_{[c_1,c_2)}$ denotes the set $\{ c\in \mathbb{R} : c_1 \leq c < c_2\}$. Let $\mathbb{R}^{n}$ denote the set of real column vectors of dimension $n$. Let $\mathbb{R}^{n \times m}$ denote the set of real $n \times m$ matrices. For a vector $x \in \Rset^n$, $x_{[i]}$ denotes the ith element of $x$. The notation $0_{n \times m}$ denotes the $n \times m$ zero matrix and $I_n$ denotes the $n \times n$ identity matrix. Let $\|\cdot\|_2$ denote the 2-norm. The Nabla symbol $\nabla$ denotes the gradient operator. For a vector $x \in \Rset^n$ and a mapping $\Phi:\Rset^n \rightarrow \Rset$
\begin{equation}
\nabla_x \Phi(x) = \begin{bmatrix}
\frac{\partial \Phi(x)}{\partial x_{[1]}} & \frac{\partial \Phi(x)}{\partial x_{[2]}} & \ldots & \frac{\partial \Phi(x)}{\partial x_{[n]}}
\end{bmatrix} .
\end{equation}
Let $\cB (x_0, r)$ denote the open ball $\{x \in \Rset^n : \|x -x_0 \|_2 < r \}$.


\section{PROBLEM FORMULATION}\label{sec:probform}
Consider a quadratic optimization problem with nonlinear equality constraints:
\begin{problem}\label{prb:originalprob}
	\begin{align}
	& \min_{z}  z^T P z + 2 q^T z  \nonumber\\
	\text{subject to} & \nonumber\\
	& G(z)=0_{m \times 1} , \nonumber
	\end{align}
\end{problem}
where $z \in \mathbb{R}^n$, $q \in \mathbb{R}^n$, $P \in \mathbb{R}^{n \times n}$ is a positive definite matrix and $G$ is a nonlinear mapping from $\mathbb{R}^n$ to $\mathbb{R}^m$. Here, $n$ is the number of optimization variables and $m$ is the number of constraints. In this paper, we are only interested in the case when the constraint set has an infinite number of points, i.e. $m<n$, since the other cases are trivial.

Let us define
\begin{equation} \label{eqn:ztox}
x = P^{1/2} z + P^{-1/2} q ,
\end{equation}
where $x\in \mathbb{R}^n$. It follows that
\begin{equation} \label{eqn:xtoz}
z = P^{-1/2} (x - P^{-1/2} q) .
\end{equation}
The constraint can be rewritten in term of $x$ as follows
\begin{equation}
G(z) = G(P^{-1/2} (x - P^{-1/2} q)) =: F(x) .
\end{equation}
It is obvious that $F$ is also a mapping from $\mathbb{R}^n$ to $\mathbb{R}^m$. The cost function is also rewritten as a function of $x$ as follows
\begin{equation}
z^T P z + 2 q^T z = x^T x - q^T P^{-1} q .
\end{equation}
Since the optimization problem is not affected by the affine term $ q^T P^{-1} q$, Problem~\ref{prb:originalprob} reduces to
\begin{problem}\label{prb:prob}
	\begin{align}
	& \min_{x}  x^T x \nonumber\\
	\text{subject to} & \nonumber\\
	& F(x)=0_{m \times 1} . \nonumber
	\end{align}
\end{problem}

The above derivation shows that any quadratic optimization problems in the form of Problem~\ref{prb:originalprob} can be transformed into the form of Problem~\ref{prb:prob}. Therefore, instead of solving Problem~\ref{prb:originalprob}, we can solve Problem~\ref{prb:prob} to get the solution in $x$. The solution in $z$ can then be deduced from~\eqref{eqn:xtoz}.

\section{PROPOSED ALGORITHM}\label{sec:algorithm}

This section first reviews the standard approach to solve a quadratic optimization problem with nonlinear equality constraints using Lagrange multipliers and Newton's method. This method will be referred to as the Lagrange-Newton method for short. Then, the proposed algorithm for solving Problem~\ref{prb:prob} is presented.

\subsection{Lagrange-Newton method}
A common approach to solve the optimization problem~\ref{prb:prob} is to solve the set of optimality conditions, which are known as the KKT conditions, using the Newton's method. 

First, define the Lagrange function
\begin{equation}
L(x,\lambda) = x^T x + \lambda^T F(x) ,
\end{equation}
where $\lambda\in \mathbb{R}^m$ is the vector of Lagrange multipliers. The optimality conditions of Problem~\ref{prb:prob} are:
\begin{align}\label{eqn:KKTcons}
\nabla_{[x , \lambda]^T} L(x,\lambda)^T &= \begin{bmatrix}
\nabla_{x} L(x,\lambda) & \nabla_{\lambda} L(x,\lambda)
\end{bmatrix}^T \nonumber \\
&= \begin{bmatrix}
2x + \nabla_x F(x)^T \lambda \\
F(x)
\end{bmatrix} = 0_{(n+m)\times 1} .
\end{align}
The system of equations~\eqref{eqn:KKTcons} is solved iteratively using Newton's method. Let $x_k$ and $\lambda_k$ be the estimated values of $x$ and $\lambda$ at any iteration $k\in \Nset$. Then the new estimated values of $x$ and $\lambda$ are computed as follows:
\begin{align}~\label{eqn:LagrangeNewton}
& \begin{bmatrix}
x_{k+1} \\ \lambda_{k+1}
\end{bmatrix} \nonumber\\&= \begin{bmatrix}
x_{k} \\ \lambda_{k}
\end{bmatrix} - \left(\nabla^2_{[x_k , \lambda_k]^T} L(x_k,\lambda_k)\right)^{-1} \nabla_{[x_k , \lambda_k]^T} L(x_k,\lambda_k)^T ,
\end{align}
where $\nabla^2_{[x_k , \lambda_k]^T} L(x_k,\lambda_k)$ is the Hessian matrix of $L(x_k,\lambda_k)$:
\begin{equation}
\nabla^2_{[x , \lambda]^T} L(x,\lambda) = \begin{bmatrix}
\nabla^2_{x} L(x,\lambda) & \nabla_x F(x)^T \\
\nabla_x F(x) & 0_{m \times m}
\end{bmatrix} ~.
\end{equation}

The iterative process stops when a predefined accuracy $\varepsilon$ is reached, i.e.
\begin{equation}\label{eqn:stopcondition}
\|F(x_k)\|_2 \leq \varepsilon,
\end{equation}
where $\varepsilon$ is a small positive value.

This algorithm converges locally to a local minimum of Problem~\ref{prb:prob} and the convergence rate is quadratic. The proof of convergence of the Newton's method can be found in~\cite{Kantorovicbook,Rallbook}. However, this method introduces additional variables, i.e. Lagrange multipliers, which increases the complexity of the problem. In addition, the need to evaluate and store the Hessian matrix increases the computational time and the amount of memory needed.

\subsection{Proposed algorithm}
In this section, we propose a new algorithm which does not introduce additional optimization variables and does not require computation of the Hessian matrix. Instead of formulating the optimality conditions~\eqref{eqn:KKTcons}, we directly solve the system of equations
\begin{equation} \label{eqn:xcons}
	F(x)=0_{m \times 1}.
\end{equation}

The proposed method to solve~\eqref{eqn:xcons} is inspired by the classical Newton's method. The original idea of the iterative Newton's algorithm is to linearize $F(x)$ around the current guess $x_k$ at iteration $k$:
\begin{equation} \label{eqn:linearizeF}
F(x) \approx \tilde{F}(x):= F(x_k) + \nabla F(x_k) (x - x_k) .
\end{equation}
Then the new guess $x_{k+1}$ of the solution has to satisfy
\begin{equation} \label{eqn:linFeq0}
\tilde{F}(x_{k+1})=0_{m \times 1} ,
\end{equation}
which is equivalent to
\begin{equation} \label{eqn:linFeq0equi}
\nabla F (x_k) x_{k+1} = \nabla F(x_k) x_k - F(x_k) .
\end{equation}

Since the system of equations~\eqref{eqn:linFeq0} is underdetermined, i.e. $m<n$, it has an infinite number of solutions. A feasible solution of~\eqref{eqn:linFeq0} is
\begin{equation} \label{eqn:generalizedNewton}
x_{k+1}= x_k - \nabla F(x_k)^{-R} F(x_k), ~ k=0,1,\ldots,
\end{equation}
where $\nabla F(x_k)^{-R}$ is the minimum 2-norm generalized inverse, or the right inverse of $\nabla F(x_k)$~\cite{Rao1972}:
\begin{equation}
\nabla F(x_k)^{-R} = \nabla F(x_k)^T (\nabla F(x_k) \nabla F(x_k)^T)^{-1} .
\end{equation}
The update~\eqref{eqn:generalizedNewton} is known as the generalized Newton's method for underdetermined systems of equations, which provides a feasible solution of~\eqref{eqn:linFeq0}. However, instead of just computing a feasible solution, we can compute the minimum 2-norm solution of~\eqref{eqn:linFeq0}, which is the solution that has smallest value of $x_{k+1}^T x_{k+1}$, i.e.
\begin{equation} \label{eqn:newiter1}
x_{k+1} = \nabla F(x_k)^{-R} (\nabla F(x_k) x_k - F(x_k)) ,
\end{equation}
or, equivalently
\begin{equation} \label{eqn:newiter2}
x_{k+1} = \nabla F(x_k)^{-R} \nabla F(x_k) x_k -  \nabla F(x_k)^{-R} F(x_k) .
\end{equation}

The generalized Newton's method~\eqref{eqn:generalizedNewton} only provides a feasible solution of the system of equations~\eqref{eqn:xcons}, but its local convergence has been proven in~\cite{BENISRAEL1966,Levin2001}. On the other hand, the new proposed iteration~\eqref{eqn:newiter2} provides the minimum 2-norm solution of~\eqref{eqn:xcons}, but its convergence has not been proven yet. From the authors' experience, sometimes iteration~\eqref{eqn:newiter2} does not converge even if the initial guess is very close to the minimum 2-norm solution of~\eqref{eqn:xcons}. To solve this issue of~\eqref{eqn:newiter2}, we propose a new iteration, which is the interpolation between iteration~\eqref{eqn:generalizedNewton} and iteration~\eqref{eqn:newiter2}, i.e.
\begin{IEEEeqnarray}{rCl} \label{eqn:proposediter}
x_{k+1} = \alpha x_k + (1-\alpha) \nabla && F(x_k)^{-R} \nabla F(x_k) x_k \nonumber \\
&& -  \nabla F(x_k)^{-R} F(x_k) ,
\end{IEEEeqnarray}
where $0 < \alpha < 1$. The proposed iteration~\eqref{eqn:proposediter} inherits the local convergence property of~\eqref{eqn:generalizedNewton}, while it converges to the minimum 2-norm fixed point of~\eqref{eqn:newiter2}, as it will be outlined next.

In what follows we will prove that iteration~\eqref{eqn:proposediter} converges locally to a fixed point that satisfies the KKT optimality conditions~\eqref{eqn:KKTcons}.

First, we introduce simplified notation for brevity. Let us denote 
\begin{equation}
J_k := J(x_k)=\nabla F(x_k)
\end{equation}
as the Jacobian matrix of $F(x_k)$, and 
\begin{equation}
T_k:=T(x_k)=\nabla F(x_k)^{-R}=J_k^T(J_k J_k^T)^{-1}
\end{equation}
as the right inverse of $J_k$. Here, $J_k\in \mathbb{R}^{m\times n}$ and $T_k\in \mathbb{R}^{n\times m}$. It follows that:
\begin{equation}
J_k T_k = I_m .
\end{equation}
Assume that $x^*$ is a fixed point of~\eqref{eqn:proposediter}, let us denote
\begin{align}
F^* &:= F(x^*), \\
J^* &:= J(x^*), \\
T^* &:= T(x^*).
\end{align} 
With the simplified notation, the proposed iteration~\eqref{eqn:proposediter} can be rewritten as:
\begin{equation} \label{eqn:proposediter1}
x_{k+1} = \alpha x_k + (1-\alpha) T_k J_k x_k - T_k F_k .
\end{equation}

\begin{lemma} \label{lem:meanvalue}
Let $\cD$ be a convex subset of $\Rset^n$ in which $F:\cD \rightarrow \mathbb{R}^m$ is differentiable and $J(x)$ is Lipschitz continuous for all~$x \in \cD$, i.e. there exists a $\gamma>0$ such that
\begin{equation}
\| J(x)-J(y) \|_2 \leq \gamma \|x-y \|_2 , \text{ for all } x,y \in \cD .
\end{equation}
Then
\begin{equation}
\| F(x)-F(y)-J(y)(x-y) \|_2 \leq \frac{\gamma}{2} \|x-y \|_2^2 , \text{ for all } x,y \in \cD .
\end{equation}
\end{lemma}

\begin{proof}
This result can be proved using the mean value theorem:
\begin{IEEEeqnarray}{rCl} 
	& & \| F(x)-F(y)-J(y)(x-y) \|_2 \nonumber\\
	&=& \| \int_{0}^{1} \left(J(y + \phi(x-y))-J(y)\right) (x-y) d\phi \|_2 \nonumber \\
	& \leq & \int_{0}^{1} \|  \left(J(y + \phi(x-y))-J(y)\right) (x-y)  \|_2 d\phi \nonumber \\
	& \leq & \int_{0}^{1} \|  \left(J(y + \phi(x-y))-J(y)\right) \|_2 \|(x-y) \|_2 d\phi \nonumber \\
	& \leq & \int_{0}^{1} \gamma \phi \|(x-y) \|_2^2 d\phi \nonumber \\
	&=& \frac{\gamma}{2} \|x-y \|_2^2 .
\end{IEEEeqnarray}
\end{proof}

\begin{lemma} \label{lem:normeq1}
Let $\cD$ be a subset of $\mathbb{R}^n$ where $J(x)$ and $T(x)$ are well defined. For any $x \in \cD$, it holds that
\begin{equation}
\| I_n - T(x)J(x) \|_2 = 1 .
\end{equation}
\end{lemma}

\begin{proof}
First we will prove that $(T(x)J(x))$ is a Hermitian matrix. We have:
\begin{IEEEeqnarray}{rCl} 
(T(x)J(x))^T &=& (J(x)^T (J(x)J(x)^T)^{-1} J(x))^T \nonumber \\
&=& J(x)^T (J(x)J(x)^T)^{-1} J(x) \nonumber \\
&=& T(x)J(x) .
\end{IEEEeqnarray}

Next we will prove that $(I_n- T(x)J(x))$ is an idempotent matrix:
\begin{IEEEeqnarray}{rCl} 
& &(I_n - T(x)J(x)) (I_n - T(x)J(x)) \nonumber \\
&= & I_n - 2 T(x)J(x) + T(x)J(x)T(x)J(x) \nonumber \\
&= & I_n - T(x)J(x) .
\end{IEEEeqnarray}

Denote $A := I_n - T(x)J(x)$. Then $A$ is a Hermitian idempotent matrix, i.e.
\begin{align}
A^T &=A ,\\
AA &=A .
\end{align}

Let $\rho_i$ be an eigenvalue of $A^T A$, then $\rho_i$ satisfies
\begin{align}
	\rho_i x &= A^T A x \nonumber \\
	&= A^T A (A^T A x ) \nonumber \\
	&= A^T A (\rho_i x) \nonumber \\
	&= \rho_i (A^T A x) \nonumber \\
	&= \rho^2_i  x, ~\forall x \in \cD .
\end{align}
This results in
\begin{equation}
	\rho_i = \rho_i^2 ,
\end{equation}
which means that $\rho_i$ is equal to either 0 or 1. If $A^T A$ is a nonzero matrix then at least one eigenvalue is 1. The 2-norm of $A$ is defined as the maximum singular value of $A$. This is equal to the square root of the maximum eigenvalue of $A^T A$, which is equal to 1.
\end{proof}

\begin{theorem}   \label{thm:optimality}
	If algorithm~\eqref{eqn:proposediter1} converges to a fixed point $x^*$, then $x^*$ satisfies the KKT optimality conditions~\eqref{eqn:KKTcons}.
\end{theorem}

\begin{proof}  
	Left-multiplying~\eqref{eqn:proposediter1} with $T_k J_k$, we have
	\begin{equation}\label{eqn:TJxkp1}
	T_k  J_k  x_{k+1} = \alpha T_k J_k x_k + (1-\alpha) T_k J_k x_k - T_k F_k  .
	\end{equation}
	Subtracting~\eqref{eqn:TJxkp1} from~\eqref{eqn:proposediter1} results in
	\begin{align}
	x_{k+1} &= \alpha x_k - \alpha T_k J_k x_k + T_k J_k x_{k+1} .
	\end{align}
	Consequently, if algorithm~\eqref{eqn:proposediter1} converges to a fixed point $x^*$ then
	\begin{align}\label{eqn:fixedpoint1}
	x^* &= \alpha x^* -\alpha T^* J^* x^* +T^* J^* x^* \nonumber\\
	&= \alpha x^* + (1-\alpha) T^* J^* x^*.
	\end{align}
	Besides, due to~\eqref{eqn:proposediter1}, $x^*$ also satisfies
	\begin{equation}\label{eqn:fixedpoint2}
	x^* = \alpha x^* + (1-\alpha) T^* J^* x^* - T^*  F^*.
	\end{equation}
	Subtracting~\eqref{eqn:fixedpoint2} from~\eqref{eqn:fixedpoint1} results in
	\begin{equation}\label{eqn:TFeq0}
	T^*  F^* = 0_{n \times 1}.
	\end{equation}
	Here, $T^*$ is a $n\times m$ matrix where $m<n$. This means that $T^* $ has an empty null space. Therefore, \eqref{eqn:TFeq0} is equivalent to 
	\begin{equation}\label{eqn:KKTcons2}
	F^*  = 0_{m \times 1}.
	\end{equation}
	
	From~\eqref{eqn:fixedpoint1} and the fact that $\alpha<1$, it follows that
	\begin{equation}\label{eqn:fixedpoint3}
	x^* =  T^* J^* x^* ,
	\end{equation}
	which is equivalent to
	\begin{equation}\label{eqn:fixedpoint4}
	(I_n -  T^* J^*) x^* = 0_{n\times 1}.
	\end{equation}
	This means that $x^*$ is in the null space of $(I_n-T^* J^*)$, which is the range of $J^{*T}$~\cite{Bernstein2009matrix}. Therefore, there exists a vector $\lambda \in \Rset^m$ such that
	\begin{equation}
	x^* = -\frac{1}{2} J^{*T} \lambda ,
	\end{equation}
	or equivalently
	\begin{equation}\label{eqn:KKTcons1}
		2x^* + J^{*T} \lambda =0_{n \times 1} .
	\end{equation}
	
	From~\eqref{eqn:KKTcons2} and~\eqref{eqn:KKTcons1}, we conclude that $x^*$ satisfies the KKT optimality conditions~\eqref{eqn:KKTcons}.	
\end{proof}

\begin{theorem} \label{thm:convergence}
Let $\cD \subseteq \Rset^n$ be an open convex invariant set for~\eqref{eqn:proposediter1} in which the following conditions hold
\begin{enumerate}[(i)]
\item $F(x)$ is Lipschitz continuous,
\item $J(x)$ is well defined and Lipschitz continuous,
\item $T(x)$ is well defined and bounded,
\item $T(x)J(x)$ and $T(x)J(x)x$ are Lipschitz continuous,
\item there exists a solution $x^* $ of the KKT optimality conditions~\eqref{eqn:KKTcons} in~$\cD$.
\end{enumerate}
Then there exist a $\alpha\in\Rset_{(0,1)}$ and a $r \in \Rset_{>0}$ such that $\cB(x^*,r) \subseteq \cD$ and iteration~\eqref{eqn:proposediter1} converges to a fixed point that satisfies~\eqref{eqn:KKTcons} for any initial condition $x_0 \in \cB(x^*,r)$.
\end{theorem}

\begin{proof}
Left-multiplying~\eqref{eqn:proposediter1} with $J_k$ yields that
\begin{equation}
J_k(x_{k+1}-x_k)=-F_k.
\end{equation}

We have
\begin{IEEEeqnarray}{rCl} 
&& x_{k+1}-x_{k} \nonumber \\
&=& \alpha (x_{k}-x_{k-1}) + (1-\alpha)(T_k J_k x_k-T_{k-1} J_{k-1} x_{k-1}) \nonumber \\
&&- (T_k F_k-T_{k-1} F_{k-1}) \nonumber \\
&=& \alpha (x_{k}-x_{k-1}) + (1-\alpha)(T_k J_k x_k-T_{k-1} J_{k-1} x_{k-1}) \nonumber \\
&&- T_k F_k-T_{k-1}J_{k-1}(x_k-x_{k-1}) \nonumber \\
&=& \alpha (I_n-T_{k-1}J_{k-1})(x_{k}-x_{k-1}) \nonumber \\
&&+(1-\alpha)(T_k J_k- T_{k-1}J_{k-1})x_k - T_k F_k  . \label{eqn:difference}
\end{IEEEeqnarray}

Due to the condition that $T(x)$ is bounded for all $x \in \cD$ and Lemma~\ref{lem:meanvalue}, we have:
\begin{align}
\|T_k F_k \|_2 \leq& \| T_{k} \|_2 \|F_k \|_2 \nonumber \\
= &  \|T_{k} \|_2 \| F_k-F_{k-1}-J_{k-1}(x_k-x_{k-1})\|_2 \nonumber \\
\leq& L \|x_k-x_{k-1}\|_2^2 , \label{eqn:ineq1} 
\end{align}
where $L>0$.

Due to the condition that $T(x)J(x)$ is Lipschitz continuous for all $x \in \cD$, it follows that:
\begin{equation}\label{eqn:ineq2}
\| (T_k J_k - T_{k-1}J_{k-1})x_k \|_2 \leq N \|x_k-x_{k-1}\|_2 ,
\end{equation}
where $N>0$.

Lemma~\ref{lem:normeq1} states that $\| I_n - T_{k-1}J_{k-1} \|_2 = 1$. Consequently:
\begin{IEEEeqnarray}{rCl} 
&& \| (I_n-T_{k-1}J_{k-1})(x_{k}-x_{k-1}) \|_2 \nonumber\\
& \leq & \| (I_n-T_{k-1}J_{k-1}) \|_2 \| (x_{k}-x_{k-1}) \|_2 \nonumber \\
& = & \|x_{k}-x_{k-1} \|_2 . \label{eqn:ineq31}
\end{IEEEeqnarray}
The equality happens if and only if $(x_{k}-x_{k-1}) $ is in the range of $(I_n-T_{k-1}J_{k-1})$, which is the null space of $J_{k-1}$~\cite{Bernstein2009matrix}. This is equivalent to:
\begin{equation}
J_{k-1}(x_{k}-x_{k-1})=-F_{k-1}=0_{m \times 1} .
\end{equation}
This shows that the equality happens if and only if $x_{k-1}$ is an exact solution of~\eqref{eqn:xcons}. Let us assume that $F_{k-1} \neq 0_{m \times 1}$ for any $k< \infty$. Then there exists a constant $M\in \Rset_{(0,1)}$ such that
\begin{equation}\label{eqn:ineq3}
\| (I_n-T_{k-1}J_{k-1})(x_{k}-x_{k-1}) \|_2 \leq M \|x_{k}-x_{k-1} \|_2 .
\end{equation}

From~\eqref{eqn:difference}, \eqref{eqn:ineq1}, \eqref{eqn:ineq2} and~\eqref{eqn:ineq3}, it follows that
\begin{IEEEeqnarray}{rCl} 
\|x_{k+1}-x_{k}\|_2 &\leq & K \|x_k-x_{k-1}\|_2 + L \|x_k-x_{k-1}\|_2^2 \nonumber \\
&=& ( K  + L \|x_k-x_{k-1}\|_2)\|x_k-x_{k-1}\|_2 , \nonumber\\ \label{eqn:total_ineq}
\end{IEEEeqnarray}
where $K=\alpha M + (1-\alpha) N $. If $N\leq 1$ then $K<1$ for all $\alpha \in \Rset_{(0,1)}$. If $N>1$ then $K<1$ if and only if
\begin{equation}\label{eqn:alpharange}
\frac{N-1}{N-M} < \alpha <1
\end{equation}
Therefore, there always exists an $\alpha \in \Rset_{(0,1)}$ such that $K<1$.

From~\eqref{eqn:proposediter1} and the conditions that $F(x)$ and $T(x)J(x)x$ are Lipschitz continuous, we have
\begin{IEEEeqnarray}{rCl} 
\| x_1-x_0 \|_2 &=& \|(1-\alpha) (T_0 J_0 x_0 - x_0) - T_0 F_0 \|_2 \nonumber \\
&=& \| (1-\alpha) (T_0 J_0 x_0 - x_0 - T^* J^* x^* + x^*) \nonumber\\
&& - (T_0 F_0 - T_0 F^*)\|_2 \nonumber \\
&=& \| (1-\alpha) (T_0 J_0 x_0 - T^* J^* x^*) \nonumber\\
&& -(1-\alpha) (x_0 - x^*) - T_0 (F_0 - F^*)\|_2 \nonumber\\
&=& \| (1-\alpha) (T_0 J_0 x_0 - T^* J^* x^*) \|_2 \nonumber\\
&& +\|(1-\alpha) (x_0 - x^*)\|_2 + \|T_0(F_0 - F^*)\|_2 \nonumber\\
& \leq & Q \|x_0 - x^*\|_2 ,
\end{IEEEeqnarray}
where $Q>0$. If $x_0$ is close enough to $x^*$ such that
\begin{equation} \label{eqn:ballradius}
\|x_0 - x^*\|_2 < \frac{1-K}{Q L} ,
\end{equation}
then it follows that
\begin{equation} \label{eqn:induction0}
K  + L \|x_1-x_{0}\|_2 < 1 .
\end{equation}
Next, we will prove that if
\begin{equation} \label{eqn:induction1}
K  + L \|x_{k+1}-x_{k}\|_2 < 1 ,
\end{equation}
then
\begin{equation} \label{eqn:induction2}
K  + L \|x_{k+2}-x_{k+1}\|_2 < 1 .
\end{equation}
Indeed, if~\eqref{eqn:induction1} holds then due to~\eqref{eqn:total_ineq} we have
\begin{equation} 
\|x_{k+2}-x_{k+1}\|_2 < \|x_{k+1}-x_{k}\|_2 .
\end{equation}
This leads to
\begin{equation}
K  + L \|x_{k+2}-x_{k+1}\|_2 < K  + L \|x_{k+1}-x_{k}\|_2 < 1 .
\end{equation}
We have proved that if~\eqref{eqn:induction1} holds then~\eqref{eqn:induction2} holds. Since~\eqref{eqn:induction0} also holds for any $x_0$ which satisfies~\eqref{eqn:ballradius}, it follows by induction that
\begin{equation}
K  + L \|x_{k+1}-x_{k}\|_2 < 1 ,~ \forall k=0,1,\ldots.
\end{equation}
Therefore, it follows from~\eqref{eqn:total_ineq} that
\begin{equation}
\|x_{k+2}-x_{k+1}\|_2 < \|x_{k+1}-x_{k}\|_2 ,~ \forall k=0,1,\ldots.
\end{equation}
Therefore, algorithm~\eqref{eqn:proposediter1} converges, and by Theorem~\ref{thm:optimality}, it converges to a solution of~\eqref{eqn:KKTcons} for any $x_0 \in \cB(x^*,r)$, where
\begin{equation}
r=\frac{1-K}{QL} ,
\end{equation}
and $\alpha$ satisfies~\eqref{eqn:alpharange}.
\end{proof}

In the case when $\cB\left(x^*,\frac{1-K}{QL}\right)$ does not lie completely inside~$\cD$, then $r$ can be chosen as the solution of
\begin{align*}
& \max_{r} r \nonumber\\
\text{subject to} & \nonumber\\
& \cB(x^*,r) \subseteq \cD .
\end{align*}

It is worth mentioning that the assumptions on Lipschitz continuity of $F(x)$, $J(x)$ and boundedness of $T(x)$ are commonly used in proving convergence of Newton-based algorithms. Different authors also use different additional assumptions such that the convergence holds, see for example~\cite{BENISRAEL1966,Haubler1986,Nashed93,Chen1997,Levin2001}.


\begin{remark}
	In the case when the equality constraint~\eqref{eqn:xcons} is nonconvex, the fixed point of both the Lagrange-Newton method and the proposed method can be either a local minimum or a local maximum. To determined whether it is a local minimum or a local maximum, it is necessary to check the second-order conditions or to check the value of the cost function in the vicinity of the fixed point. How to guarantee convergence of the proposed algorithm to a global minimum will be the subject of future research.
\end{remark}

In summary, similar to the classical Lagrange-Newton algorithm, the proposed algorithm converges locally to a KKT point of Problem~\ref{prb:prob}. However, the proposed algorithm requires less computational cost and less amount of memory than the Lagrange-Newton algorithm. This is beneficial for large-scale problems and real-time applications.

\section{EXAMPLES}\label{sec:example}
This section presents two examples to verify the performance of the proposed algorithm.

\subsection{Least square fitting of ellipses}
Fitting of ellipses to data points is a fundamental task in pattern recognition and computer vision. This problem has been extensively studied and widely applied. In this example, we follow the direct least square fitting method proposed in~\cite{Fitzgibbon1999}.

An ellipse in 2D $(x,y)$-coordinate can be represented by a second order polynomial:
\begin{equation}
\Gamma(\theta,\eta) =\theta^T \eta= ax^2 + bxy + c y^2 +dx+ey+f =0 ,
\end{equation}
where $\theta$ is the coefficients vector:
\begin{equation}
\theta = \begin{bmatrix}
a&b&c&d&e&f
\end{bmatrix}^T ,
\end{equation}
and
\begin{equation}
\eta = \begin{bmatrix}
x^2&xy&y^2&x&y&1
\end{bmatrix}^T .
\end{equation}
The polynomial $\Gamma(\theta,\eta)$ is called the algebraic distance of a point $(x, y)$ to the ellipse~$\Gamma(\theta,\eta)=0$. An approach to fit an ellipse to $N_p$ data points $(x_i,y_i)$, $i=1,\ldots,N_p$, is by minimizing the sum of squared algebraic distances:
\begin{equation} \label{eqn:ellipseobj}
\sum_{i=1}^{N_p} \Gamma(\theta,\eta_i)=\theta^T D^T D \theta ,
\end{equation}
where
\begin{equation}
D= \begin{bmatrix}
\eta_1 & \eta_2 &\cdots& \eta_{N_p}
\end{bmatrix}^T .
\end{equation}

In addition, an ellipse has to satisfy the constraint that the discriminant $b^2-ac$ is negative. Since we have the freedom to arbitrarily scale the parameters, the constraint can be written as
\begin{equation}\label{eqn:ellipsecons}
\theta^T S \theta = -1,
\end{equation}
where 
\begin{equation}
S=\begin{bmatrix}
0&0&-2&0&0&0\\
0&1&0&0&0&0\\
-2&0&0&0&0&0\\
0&0&0&0&0&0\\
0&0&0&0&0&0\\
0&0&0&0&0&0\\
\end{bmatrix} .
\end{equation}

From~\eqref{eqn:ellipseobj} and~\eqref{eqn:ellipsecons}, the ellipse fitting problem can be formulated as
\begin{problem}\label{prb:fitellipse}
	\begin{align*}
	& \min_{\theta}  \theta^T R \theta \nonumber\\
	\text{subject to} & \nonumber\\
	& \theta^T S \theta +1=0 ,
	\end{align*}
\end{problem}
where $R=D^TD$ is positive definite. Following the transformation derived in Section~\ref{sec:probform}, Problem~\ref{prb:fitellipse} can be transformed into
\begin{problem}\label{prb:fitellipse1}
	\begin{align*}
	& \min_{\varphi}  \varphi^T \varphi \nonumber\\
	\text{subject to} & \nonumber\\
	& \varphi^T H \varphi + 1 =0 ,
	\end{align*}
\end{problem}
where
\begin{equation}\label{eqn:ellipseprbtransform}
\varphi = R^{1/2} \theta ,\qquad H = R^{-1/2} S R^{-1/2} .
\end{equation}

In this example, 11 data points are generated from the ellipse shown in Fig.~\ref{fig:ellipsefitting}. The data points are then corrupted by adding zero-mean Gaussian noise of standard deviation~$\sigma=1$. 

The fitting problem is solved using both the Lagrange-Newton algorithm and the proposed algorithm. Since the proposed algorithm can only solve problems in the form of Problem~\ref{prb:fitellipse1}, it has to spend time on transforming the original Problem~\ref{prb:fitellipse} into Problem~\ref{prb:fitellipse1}. For a fair comparison, the proposed algorithm solves Problems~\ref{prb:fitellipse1}, while the Lagrange-Newton algorithm solves directly the original Problem~\ref{prb:fitellipse}, in order to avoid wasting computation time on the transformation~\eqref{eqn:ellipseprbtransform}. 

The algorithms are implemented on a 2.4GHz computer. The predefined accuracy is~$\varepsilon=10^{-4}$. The initial point for the Lagrange-Newton algorithm is chosen as~$\theta_0=\begin{bmatrix}
1&1&1&1&1&1
\end{bmatrix}^T$, and the initial point for the proposed algorithm is chosen as~$\varphi_0=R^{1/2} \theta_0 $. For the proposed algorithm, $\alpha$ is chosen equal to $0.2$. 

Both algorithms converge to the same solution as shown in Fig.~\ref{fig:ellipsefitting}. The Lagrange-Newton algorithm converges after 4 iterations, while the proposed algorithm converges after 5 iterations. However, despite the fact that the proposed algorithm needs more iterations to converges and also has to calculate the transformation~\eqref{eqn:ellipseprbtransform}, its total computation time is 0.43ms, while the total computation time of the Lagrange-Newton algorithm is 0.62ms. If $\alpha$ is chosen equal to $0.1$ then the proposed algorithm also converges after 4 iterations like the Lagrange-Newton algorithm, and the total computation time reduces further to 0.39ms. This demonstrates the advantage of the proposed algorithm in computation time.

\begin{figure}[t]
	\centering
	\includegraphics[width=1\columnwidth]{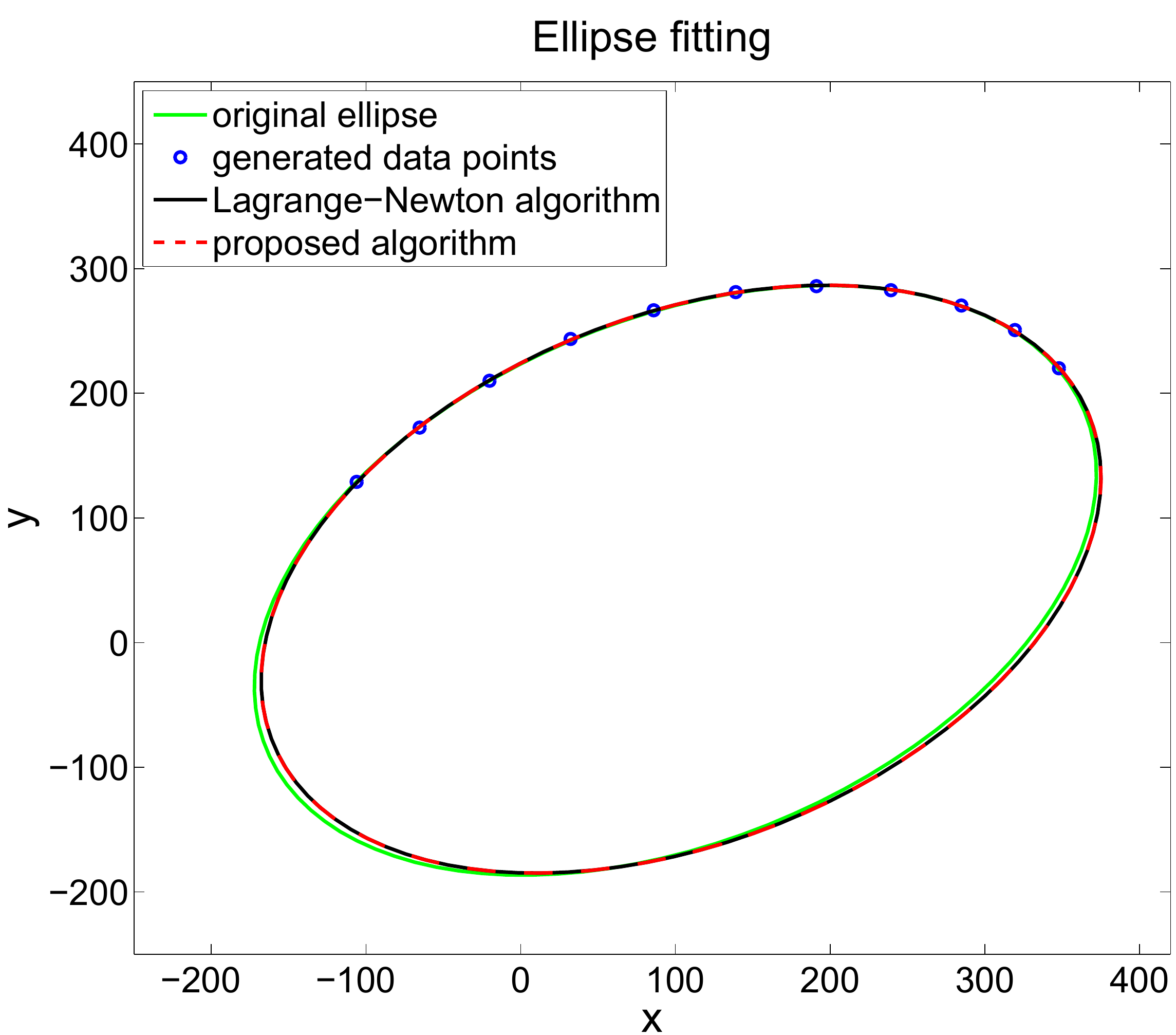}
	\caption{The original ellipse and the returned ellipses.}
	\label{fig:ellipsefitting}
\end{figure}

\subsection{State reference generation for torque control in externally excited synchronous machines}
State reference generation for torque control is a challenging problem for externally excited synchronous machines (EESM). It requires solving a quadratic optimization problem with quadratic equality constraints in real-time~\cite{Carpiuc2015}. This section solves the EESM state reference problem using the proposed algorithm.

The electromagnetic torque of a EESM is calculated as follows
\begin{equation} \label{eqn:torqeeq}
T_e = \frac{3P_p}{2}\left(M_d i_{qs} i_e + (L_d -L_q) i_{ds} i_{qs}\right) ,
\end{equation}
where $i_{ds}$ and $i_{qs}$ are the $d$- and $q$-axis stator currents, $i_e$ is the excitation current, $P_p$ is the number of pole pairs, $M_d$ is the mutual inductance between stator and rotor windings, $L_d$ and $L_q$ are the $d$- and $q$-axis inductances, respectively.

The copper losses is calculated as
\begin{equation}
P_{Cu}= \frac{3}{2} R_s(i_{ds}^2 +i_{qs}^2) + R_e i_e^2 ,
\end{equation}
where $R_s$ is the stator resistance and $R_e$ is the excitation resistance.

In torque control of EESM, the state variables can be chosen as
\begin{equation}
x = \begin{bmatrix}
i_{ds}\sqrt{\frac{3R_s}{2}} & i_{qs}\sqrt{\frac{3R_s}{2}}  & i_e \sqrt{R_e}
\end{bmatrix}^T ,
\end{equation}
and the output is the electromagnetic torque
\begin{equation}
y=T_e .
\end{equation}
The torque equation~\eqref{eqn:torqeeq} can be rewritten as
\begin{equation}
y = x^T C x ,
\end{equation}
where
\begin{equation}
C=\begin{bmatrix}
0 & \frac{P_p(L_d-Lq)}{2R_s} &0 \\
\frac{P_p(L_d-Lq)}{2R_s}  & 0 & \frac{3P_p M_d}{4} \sqrt{\frac{2}{3R_s R_e}} \\
0& \frac{3P_p M_d}{4} \sqrt{\frac{2}{3R_s R_e}} & 0
\end{bmatrix} .
\end{equation}
The copper losses can be rewritten as
\begin{equation}
P_{Cu}= x^T x .
\end{equation}

For torque control, let $y_{ref}$ be the reference torque, we have to find a value for the state variable $x$ such that the torque $y$ is equal to the reference $y_{ref}$. There is an infinite number of solutions for this problem.  When there is an infinite number of solutions, we have the freedom to choose the solution which is beneficial for the application. An attractive solution is the one that minimizes the copper losses. The state reference generation problem can be formulated as
\begin{problem}\label{prb:refgenEESM}
	\begin{align}
	& \min_{x}  x^T x \nonumber\\
	\text{subject to} & \nonumber\\
	&x^T C x - y_{ref} =0 . \nonumber
	\end{align}
\end{problem}

In this example, let us consider an EESM with the parameters given in Table~\ref{tbl:EESMparameters}. The reference torque is~$y_{ref}=10$Nm.
\begin{table}[h]
	\caption{EESM parameters}
	\label{tbl:EESMparameters}
	\begin{center}
		\begin{tabular}{|l|c|c|c|}
			\hline
			\textit{Parameter} & \textit{Symbol} & \textit{Value} & \textit{Unit} \\
			\hline
			Number of pole pairs & $P_p$ & $8$ &  \\
			Stator phase resistance & $R_s$ & $0.00775$ & $\Omega$ \\
			Excitation winding resistance & $R_e$ & $7.4$ & $\Omega$ \\
			Mutual inductance & $Md$ & $9.069$ & mH \\
			$d$-axis inductance & $L_d$ & $0.1488$ & mH \\
			$q$-axis inductance & $L_q$ & $0.2264$ & mH \\
			\hline
		\end{tabular}
	\end{center}
\end{table}

Problem~\ref{prb:refgenEESM} is solved using both the Lagrange-Newton algorithm and the proposed algorithm. The algorithms are implemented on a 2.4GHz computer. The predefined accuracy is~$\varepsilon=10^{-7}$. The initial point for both algorithms is chosen as~$x_0=\begin{bmatrix}
-1 &1&1
\end{bmatrix}^T$. For the proposed algorithm, $\alpha$ is chosen equal to $0.3$. 


Both algorithms converge after $7$ iterations. The resulting solutions are the same for both algorithms:
\begin{equation}
x^\ast =\begin{bmatrix}
-1.083 & 5.133 & 5.017
\end{bmatrix}^T .
\end{equation} 
The resulting currents are
\begin{align}
i_{ds} &=-10.046 \mathrm{(A)}, \\
i_{ds} &=47.604 \mathrm{(A)}, \\
i_e &=1.844 \mathrm{(A)}.
\end{align}
Although both algorithms converge to the same solutions, the total computational time of the Lagrange-Newton algorithm is 0.6ms, while the total computational time of the proposed algorithm is 0.3ms. This demonstrates the improvement in computational time that the proposed algorithm can bring. The algorithm is thus beneficial for real-time machines with high sampling frequency where the time for computation is limited and implementation of embedded optimization solvers is impractical.

\section{CONCLUSIONS} \label{sec:conclusions}
This paper proposed a new algorithm for solving quadratic optimization problems with nonlinear equality constraints. It was proven that the proposed algorithm converges locally to a solution of the KKT optimality conditions. The algorithm is computationally efficient since it does not introduces additional optimization variables and does not require evaluation of the Hessian matrix. The effectiveness of the proposed algorithm was demonstrated in two examples. 

For future research, how to guarantee convergence of the proposed algorithm to a global minimum is of interest. Another interesting topic is how to compute the interpolation coefficient such that the fastest speed of convergence is achieved.




%
%
\section*{ACKNOWLEDGMENT}
The authors are grateful to Dr. Sabin - Constantin Carpiuc for his help with the EESM example.

\bibliographystyle{IEEEtran} 
\bibliography{references}

\end{document}